\documentclass[11pt]{article}
\usepackage{latexsym,amsmath,stackrel,color,amsthm,amssymb,epsfig,graphicx,mathrsfs}
\usepackage{graphicx}
\usepackage{amssymb}
\usepackage[left=1in,top=1in,right=1in,bottom=1in]{geometry}
\usepackage[linktocpage=true]{hyperref}
\usepackage{setspace}
\usepackage{amssymb, amsmath, amsthm, graphicx,mathrsfs}
\usepackage{caption}

\usepackage{comment,caption}
\usepackage{tikz}
\makeatletter
\def\thm@space@setup{%
  \thm@preskip=\parskip \thm@postskip=0pt
}
\makeatother
\usepackage{thmtools}

\declaretheoremstyle[%
  spaceabove=6pt,%
  spacebelow=6pt,%
  headfont=\normalfont\itshape,%
  postheadspace=1em,%
  qed=\qedsymbol%
]{mystyle}

\def\qed{\hfill\ifhmode\unskip\nobreak\fi\quad\ifmmode\Box\else\hfill$\Box$\fi}
\def\ite#1{\hfill\break${}$\hbox to 50pt {\quad(#1)\hfill}}
\newtheorem{thm}{Theorem}[section]

\newtheorem{problem}[thm]{Problem}
\newtheorem{definition}{Definition}

\newtheorem{lem}[thm]{Lemma}
\newtheorem{conj}[thm]{Conjecture}

\def\ex{{\rm{ex}}}

\newcommand{\cir}{\circlearrowright}
\tikzstyle{vertex}=[circle,fill=black,inner sep=2pt]
\tikzstyle{vertrect}=[draw,rectangle,inner sep=2pt]
\tikzstyle{vertdia}=[draw,diamond,inner sep=2pt]

\begin{document}

\title{\vspace{-0.7in} Ordered and convex geometric trees with linear extremal function}

\author{ 
 Zolt\'an F\" uredi\thanks{Research supported by grant KH 130371 
from the National Research, Development and Innovation Office NKFIH and
by the Simons Foundation Collaboration grant \#317487.}
\and
Alexandr Kostochka\thanks{Research  supported in part by NSF grant
 DMS-1600592 and by grants 18-01-00353A  and 16-01-00499 of the Russian Foundation for Basic Research.
} 
\and
Dhruv Mubayi\thanks{Research partially supported by NSF awards DMS-1300138 and DMS-1763317.} \and Jacques Verstra\"ete\thanks{Research supported by NSF award DMS-1556524.}
}

\pagestyle{myheadings} \markright{{\small{\sc Z.~F\"uredi, A.~Kostochka, D.~Mubayi,  J. Verstra\"ete:   
 Ordered and 
   geometric trees}}}
\date{December 12, 2018}

\maketitle

\vspace{-0.4in}

\begin{abstract}
The extremal functions $\ex_{\rightarrow}(n,F)$ and $\ex_{\cir}(n,F)$
for ordered and convex geometric acyclic graphs $F$ have been extensively investigated by a number of researchers. Basic questions are to determine when $\ex_{\rightarrow}(n,F)$ and $\ex_{\cir}(n,F)$ are linear in $n$, the latter posed by Bra\ss-K\'arolyi-Valtr in 2003. In this paper, we answer both these questions for every tree  $F$.

We give a forbidden subgraph characterization for a family $\cal T$ of ordered trees with $k$ edges, and show that  $\ex_{\rightarrow}(n,T) = (k - 1)n - {k \choose 2}$ for all $n \geq k + 1$ when $T \in {\cal T}$ and $\ex_{\rightarrow}(n,T) = \Omega(n\log n)$ for $T \not\in {\cal T}$. We also describe the family ${\cal T}'$ of the convex geometric trees with linear Tur\' an number and show that for every
convex geometric tree $F\notin {\cal T}'$, $\ex_{\cir}(n,F)= \Omega(n\log \log n)$.
\end{abstract}

\begin{flushright}
 Dedicated to the memory of B. Gr\" unbaum
\end{flushright}

\section{Introduction}

An {\em ordered graph} refers to a graph whose vertex set is linearly ordered and a {\em convex geometric} or {\em cg} graph refers to a graph whose vertex set is cyclically ordered.  Throughout this paper, an $n$-vertex ordered or cg graph will be assumed to have vertex set $[n] := \{1,2,\dots,n\}$ with the natural ordering $<$; in the cg setting we use $v<w<x$ to denote that $w$ lies between $v$ and $x$ in the clockwise orientation. For $n$ a positive integer and $F$ an ordered (respectively, cg) graph, let the extremal function $\ex_{\rightarrow}(n,F)$ (respectively, $\ex_{\cir}(n,F)$)
 denote the maximum number of edges in an $n$-vertex ordered (respectively, cg) graph that does not contain $F$. Both $\ex_{\rightarrow}(n,F)$ and
 $\ex_{\cir}(n,F)$
 have been extensively studied in the literature,
in particular in the case where $F$ is a forest. To describe the known results, we require some terminology.

Given subsets $A, B$ of a linearly ordered set, write $A<B$ to denote that $a<b$ for every $a \in A$ and  $b \in B$.
The {\em interval chromatic number} $\chi_i(F)$ of an ordered graph $F$ is the minimum $k$
	such that the vertex set of $F$ can be partitioned into sets $A_1<A_2<\cdots <A_k$ such that no edge has both endpoints in any $A_i$. We call these sets {\em intervals} or {\em segments}.
It is straightforward to see that if $\chi_i(F)>2$, then  $\ex_{\rightarrow}(n,F)=\Theta(n^2)$, since an ordered  complete balanced bipartite graph with interval chromatic number two does not contain $F$.

\subsection{Ordered graphs}
Standard results in extremal graph theory imply that if $\ex(n,F) = n^{1 + o(1)}$, then $F$ is acyclic. This motivates the following central conjecture in the area, due to Pach and Tardos~\cite{PT}:

\begin{conj}\label{main}
For every forest $F$ with $\chi_i(F)=2$ we have $\ex_{\rightarrow}(n,F) = n(\log n)^{O(1)}$ as $n \rightarrow \infty$.
\end{conj}

Conjecture~\ref{main} remains open in general, though it was  verified for all $F$ with at most four edges by F\"{u}redi and Hajnal~\cite{FH}. Based on their work, Tardos~\cite{T} determined the order of magnitude of $\ex_{\rightarrow}(n,F)$ for every ordered graph $F$ with at most four edges (see Corollary 4.3 in~\cite{T}); in particular this verifies Conjecture \ref{main} for acyclic graphs with at most four edges and interval chromatic number two.
 Klazar~\cite{K} showed that the partial case of the F\"{u}redi-Hajnal conjecture that $\ex_{\rightarrow}(n,F) = O(n)$ for every matching $F$, implies
the Stanley-Wilf conjecture, which was proved by Marcus and Tardos~\cite{MT}. We also point out that extremal problems for ordered forests have applications to 
theoretical computer science, to search trees and path-compression based data structures (see Bienstock and Gy\"{o}ri~\cite{BG}, and Pettie~\cite{Pettie} for a survey).

\medskip

A particularly interesting phenomenon, discovered by F\"{u}redi and Hajnal~\cite{FH}, is that the order of magnitude of the extremal function for the
ordered forest $\{13,35,24,46\}$ consisting of two interlacing paths of length two is determined by the extremal theory for Davenport-Schinzel sequences, and in particular the extremal function has order of magnitude $\Theta(n\alpha(n))$, where $\alpha(n)$ is the inverse Ackermann function.
Further progress towards the conjecture was made by Kor\'{a}ndi, Tardos, Tomon and Weidert~\cite{KTTW}, in the equivalent reformulation of the problem in
terms of forbidden 0-1 submatrices of 0-1 matrices, giving a wide class of graphs $F$ for which $\ex_{\rightarrow}(n,F) = n^{1 + o(1)}$ as $n \rightarrow \infty$. The following basic question closely related to Conjecture~\ref{main} also has  information theoretic applications (see for instance~\cite{BG}).

\begin{problem}\label{prob1}
Determine which ordered forests have linear extremal functions.
\end{problem}
The problem is not even solved for some forests with five edges. And the above example by F\"{u}redi and Hajnal of a $4$-edge forest
with extremal function involving  the inverse Ackermann function indicates that the problem is likely to be hard.
However, it turns out that for trees the situation is simpler.

 In this paper, we resolve Problem~\ref{prob1} for ordered trees, and also determine the exact extremal function for ordered trees when the extremal function is linear. This exact result is perhaps surprising, since the situation in the unordered case is  complicated, as represented by the Erd\H os-S\'os conjecture. But the ordered situation has the benefit that most trees cannot have linear extremal function. On the other hand, the $\log n$ jump in complexity for trees with nonlinear extremal function is perhaps also interesting.

  The description of the ordered trees with linear extremal functions is based on three forbidden subtrees which are the  ordered paths $P, Q$ and $R$ shown below.

 \begin{center}
 	\includegraphics[width=2.1in]{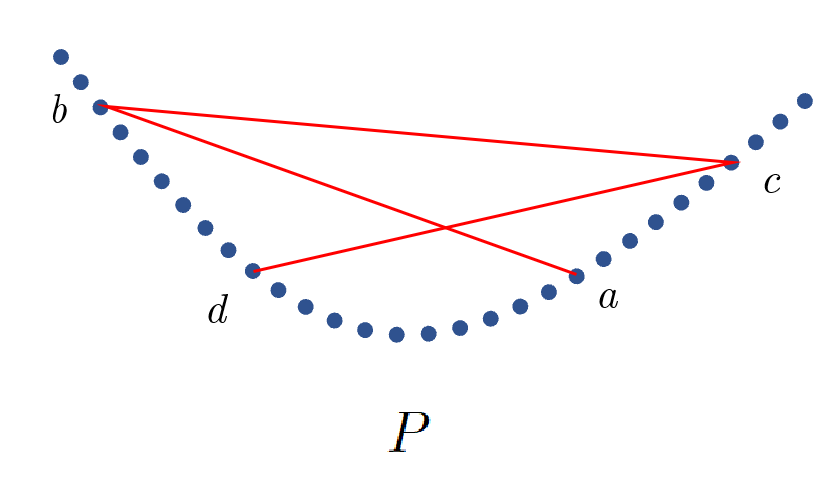}  \includegraphics[width=2.1in]{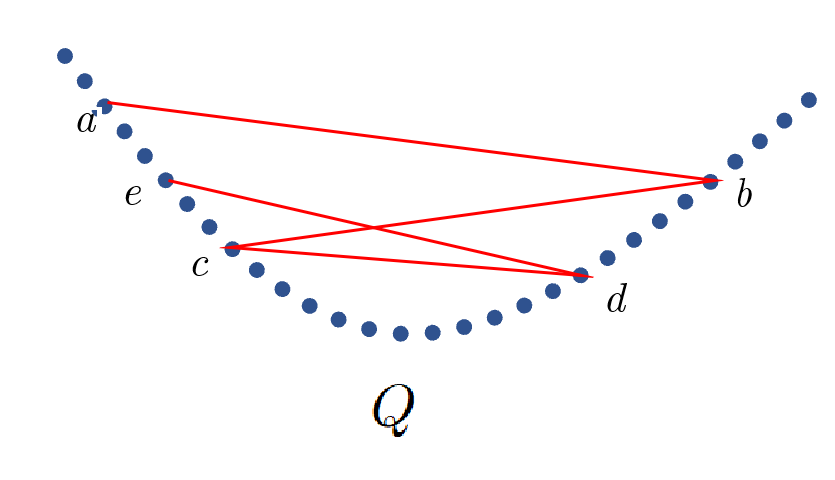} \includegraphics[width=2.1in]{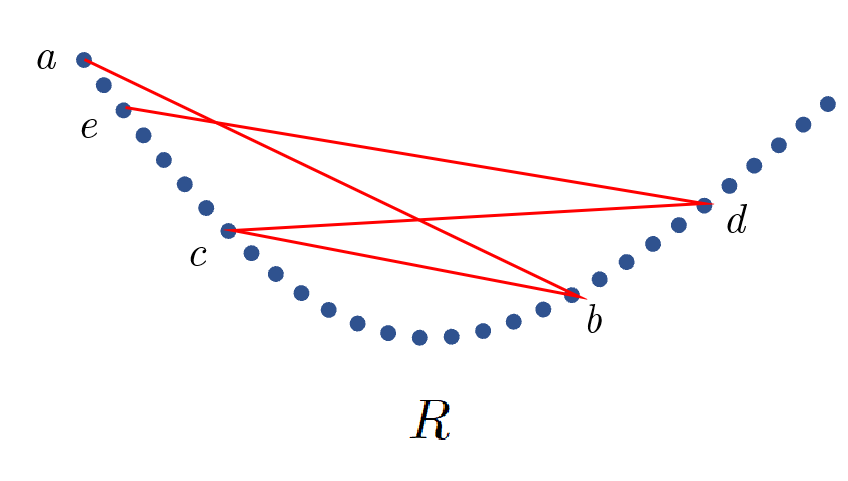}
 	
 	Figure 1 : Forbidden paths $P$, $Q$ and $R$.
 \end{center}

We are now ready to state our first main result.
\begin{thm}\label{trees}
Let $T$ be an ordered tree with $k$ edges and $\chi_i(T)=2$. If $T$ contains at least one of $P, Q, R$, then $\ex_{\rightarrow}(n,T) = \Omega(n\log n)$ as $n \rightarrow \infty$, otherwise
\[ \ex_{\rightarrow}(n,T) = (k - 1)n - {k \choose 2}\]
for all $n \ge k+1$.
\end{thm}
As a corollary to Theorem \ref{trees}, if $T$ is any ordered forest containing a path of length four with two or more crossing edges, then $\ex_{\rightarrow}(n,T) = \Omega(n\log n)$.


\subsection{Convex geometric graphs} Problem~\ref{prob1} was posed by Bra{\ss}-K\'arolyi-Valtr~\cite{BKV} in the context of convex geometric graphs, and remains open. Using our methods for ordered graphs and some modifications of constructions due to Tardos, we are able to determine all cg trees with linear extremal function. For convenience, we assume that the vertex set of any cgg we consider lies on a convex set $\Omega$ in the plane.  We say that a cgg $G$ is {\em crossing} or {\em has a crossing} if some pair of its edges intersect geometrically at a point that is not on $\Omega$.

\begin{definition}
	Let ${\cal P}=\{P^0, P^1, P^2\}$ denote the family of three cg forests each comprising two copies $P=abcd, P'=a'b'c'd'$ of a three-edge path  with the following properties:
	
		$\bullet$ the center edges $bc \in P$ and $b'c' \in P'$ do not cross each other \newline
		$\bullet$ the pair of edges $ab$ and $cd$ cross at  $p$  and the pair $a'b'$ and $c'd'$ cross at $p'$ \newline
		$\bullet$ if $\{b,c\}\ne\{b', c'\}$, then
	 $p$ and $p'$ lie outside the region whose boundary contains the segments $bc, b'c'$ and $\Omega$ while if $\{b,c\}=\{b', c'\}$, then $p$ and $p'$ lie on opposite sides of $bc=b'c'$.
	
	 	 We allow $bc$ and $b'c'$ to share $i$ endpoints where $0\le i \le 2$ and we denote the corresponding member of ${\cal P}$ by $P^i$; hence $|V(P^i)| = 8-i$. Note that $P^1$ and $P^2$ are connected while $P^0$ is not.
		\end{definition}

\begin{center}
	\includegraphics[width=6in]{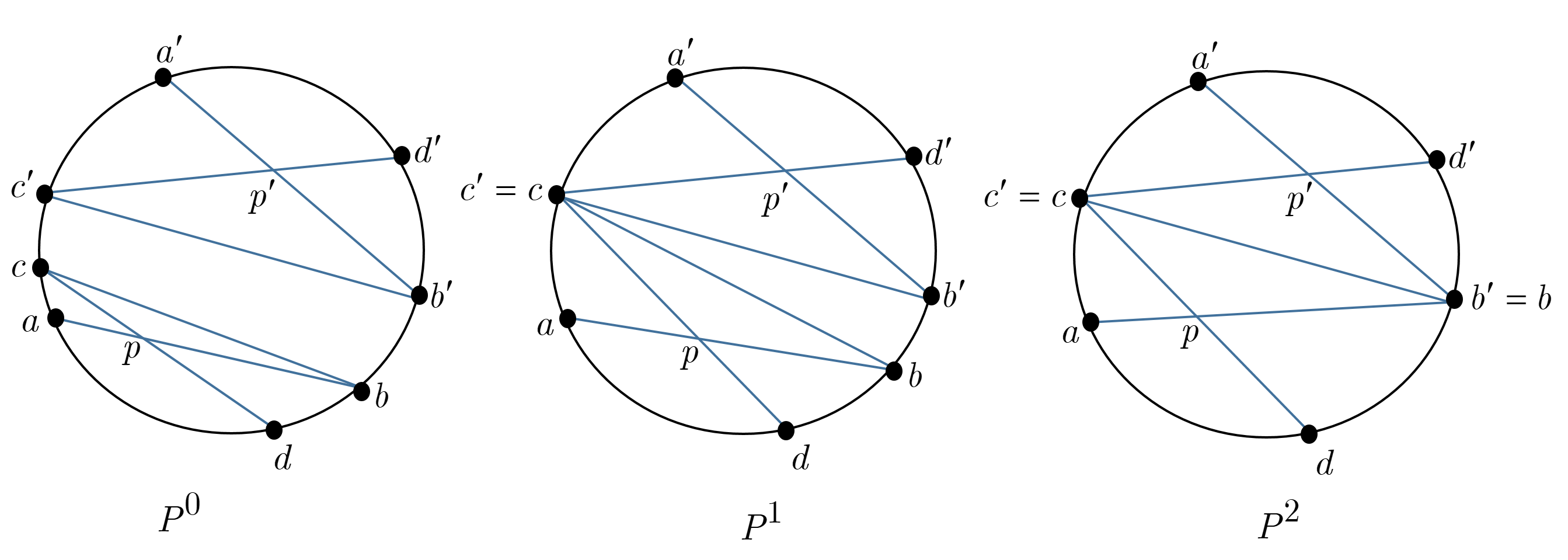}
	
	\vspace{-0.1in}
	
	Figure 2 : The family ${\cal P}$.
\end{center}

Given vertices $a_1, \ldots, a_t$ in a cyclically ordered set $\Omega$ we write $a_1<a_2<\cdots <a_t$ to mean that the vertices are encountered in the order $a_1, a_2, \ldots, a_t, a_1$ when traversing $\Omega$ in the clockwise direction.
Given subsets $A, B$ of $\Omega$, write $A<B$ to denote that there are no elements $a, a' \in A$ and $b, b' \in B$ such that $a<b<a'<b'$. In other words, the intervals $A$ and $B$ appear as disjoint arcs/intervals of $\Omega$. The definition extends naturally to more than two intervals.

 In analogy with the definition of interval chromatic number for ordered graphs,
the {\em cyclic chromatic number}  $\chi_c(G)$ of a cg graph $G$ is the minimum $k$
such that the vertex set of $G$ can be partitioned into (nonoverlapping) intervals $A_1<A_2<\cdots <A_k$ and no edge has both endpoints in any $A_i$.
It is again straightforward to see that if $\chi_c(G)>2$, then  $\ex_{\cir}(n,G)=\Theta(n^2)$.
Consequently, as we are aiming for a characterization of those $G$ for which $\ex_{\cir}(n, G)=O(n)$ we may restrict to $G$ with $\chi_c(G)=2$.
\begin{thm} \label{cgthm}
	Fix $k>2$ and let $T$ be a cg tree with $k$ edges and $\chi_c(T)=2$. Then  either  $\ex_{\cir}(n, T) = \Theta(n)$
	or $\ex_{\cir} = \Omega(n \log\log n)$ where the former holds iff $T$ contains no crossing four-edge path and no member of $\cal P$.
		\end{thm}

It is interesting to contrast Theorem~\ref{cgthm} with Theorem~\ref{trees}. As a general rule, determining $\ex_{\cir}(n, F)$ seems more difficult than determining $\ex_{\to}(n, F)$. Our experience suggests that both problems exhibit similar but different behavior. For example,  we were not able to determine  $\ex_{\cir}(n, T)$ exactly when it is $\Theta(n)$ like in the ordered case.
 One problematic cg tree is the double star $D$ with $k$ edges and maximum number of crossings. It is easy to observe that $\ex_{\cir}(n, D) \le \ex_{\to}(n, D) = O(n)$ but an exact result for $\ex_{\cir}(n, D)$ seems harder to achieve. Perhaps this is the main impediment to obtaining an exact result in Theorem~\ref{cgthm}. Also, it is not true that all $k$-edge cg trees with linear extremal function have the same extremal function, and we do not know whether every nonlinear extremal function for a cg tree grows at least as $n\log n$.

\section{Ordered trees}
In this section we prove Theorem~\ref{trees}.
	In Section \ref{nonlinear}, we give the constructions which show that each of the ordered paths $P, Q, R$ has extremal function of order at least $n\log n$.
	 In Section \ref{ztrees}, we describe the structure of the ordered trees of interval chromatic number two which do not contain $P,Q$ or $R$. Then in Section \ref{linear}, we determine the extremal function for all those trees.
\subsection{Ordered trees with nonlinear extremal function}\label{nonlinear}

The paths $P, Q$ and $R$ are displayed in Figure 1. In this section, we present for each of $P, Q$ and $R$ a construction of an $n$-vertex ordered graphs with $\Theta(n\log n)$ edges that does not contain $P,Q$ and $R$ respectively. These results are not new, and if fact Tardos~\cite{T} showed that the extremal functions for $P,Q$ and $R$ are all actually of order $n\log n$.

\medskip

{\bf Construction avoiding $P$.} We start with the simple construction that does not contain $P$: form an ordered graph on $[n]$ with edges $ij$ such that $|i - j| = 2^h$ for some $h$. This graph has $\Omega(n\log n)$ edges. It does not contain $P$, since if $V(P) = \{ad,ac,bd\}$ where $a < b < c < d$, then for some $h,i,j$, $2^h = |a - d| < |a - c| + |b - d| = 2^i + 2^j$ whereas $\max\{i,j\} < h$ implies $2^i + 2^j \leq 2^h$, a contradiction. Another was to achieve this construction is to take the graph of the $k$-dimensional cube, where we construct the graph in the usual recursive manner and $n=2^k$.

\medskip

{\bf Construction avoiding $Q$.} Bienstock and Gy\"{o}ri~\cite{BG} gave a construction showing $\ex_{\rightarrow}(n,Q) = \Omega(n\log n/\log\log n)$, and a simple construction giving $\ex_{\rightarrow}(n,Q) = \Omega(n\log n)$ was given by F\"{u}redi and Hajnal~\cite{FH}.
The construction on $2n$ vertices consists of edges between two intervals $I_n$ and $J_n$ of size $n$. For $n = 1$, we take a single edge.
Having the construction at stage $n$, with intervals $I_n$ and $J_n$ of size $n$, take four intervals $I_n,I'_n$ and $J_n,J'_n$ of length $n$,
in that order. We put the preceding construction between $I_n$ and $J_n$ and between $I'_n$ and $J'_n$, and then add a matching $M$
consisting of an edge from the $i$th vertex of $I'$ to the $i$th vertex of $J$ for $i \in [n]$. If $f(n)$ is the number of edges in the old construction, then the new construction has $2f(n) + n$ edges. We conclude $f(2n) = 2f(n) + n$ which implies $f(n) = \frac{1}{2}n\log_2 n + n$. It is shown in~\cite{FH} that this construction does not contain $P$, and so $\ex_{\rightarrow}(2n,Q) \geq \frac{1}{2}n\log_2 n + n$ for all $n \geq 1$.

\medskip

{\bf Construction avoiding $R$.} A similar type of construction avoids $R$. For $n = 1$ we again take a single edge, and having created a construction with two intervals $I_n$ and $J_n$ of length $n$, we take four intervals $I_n,I'_n$ and $J_n,J'_n$ of length $n$ in that order. We put the preceding construction between $I_n$ and $J'_n$ and between $I'_n$ and $J_n$, and then add a matching $M$ consisting of an edge from the $i$th vertex of $I$ to the $i$th vertex of $J$ for $i \in [n]$. Then the number of edges in the construction with $2n$ vertices is $f(n)$ as above,
and the construction does not contain $R$.

\subsection{Structure of trees not containing $P,Q$ or $R$}\label{ztrees}
 In this section, we consider trees which do not contain $P,Q$ or $R$,
and describe their structure. The {\em length} of an edge $ij$ with $i,j \in [n]$ is $|i - j|$. Edges $ij$ and $i'j'$ with $i < j$ and $i' < j'$ {\em cross} if $i < i' < j < j'$ or $i' < i < j' < j$.

\medskip

{\bf Increasing trees.} An {\em increasing tree} is an ordered tree of interval chromatic number two, with parts equal to intervals $I,J \subseteq [n]$,
described as follows. A single edge is an increasing tree. Given an increasing tree, with longest edge $ij$ where $i \in I$
and $j \in J$, we create an increasing tree with one more edge $i'j'$ with $i' \in I$ and $j' \in J$
by requiring $i' = i$ and $j' > j$ or $j' = j$ and $i' < i$. Note that an increasing tree has no crossing edges, and the edges
have a unique ordering by the increasing order of their lengths. Also, increasing trees do not contain $P,Q$ or $R$ since they
have no crossing edges.

\medskip

{\bf $z$-trees.} A {\em $z$-tree} is an ordered tree $Z$ with interval chromatic number two, say with parts equal to intervals $I$ and $J$, consisting
of a union of an increasing tree $T$ with longest edge $ij$ where $i \in I$ and $j \in J$, together with a set $S_j$ of edges of the form $hj$ with $h \in I$ and $h < i$ and a set $S_i$ of edges of the form $ik$ with $k \in J$ and $k > j$. These sets $S_i$ and $S_j$ are allowed to be empty.
Note that any two of the edges $hj$ and $ik$ cross. 

An example of a $z$-tree is below, where the tree increasing tree $T$ is shown in solid edges whereas $S_i$ and $S_j$ are in dashed edges.

\begin{center}
\includegraphics[width=3in]{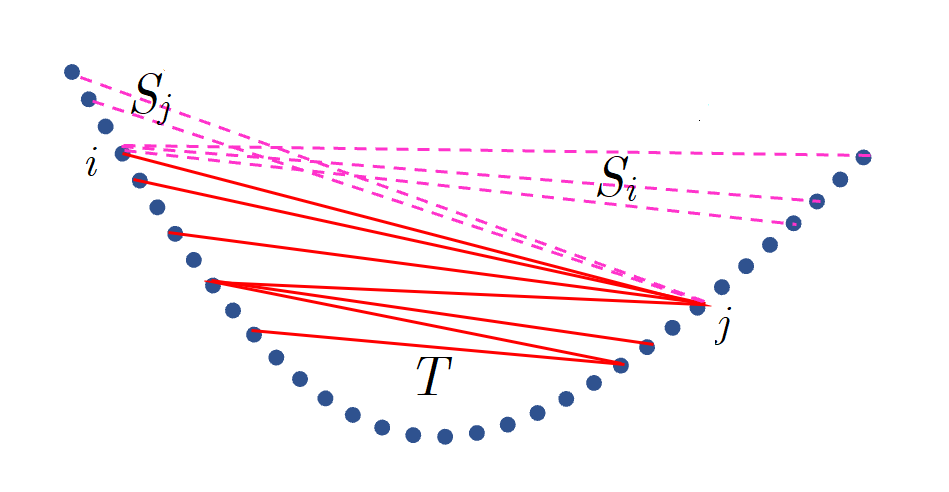}

\vspace{-0.2in}

Figure 3 : A $z$-tree.
\end{center}

Note that the partition $E(T) \cup S_i \cup S_j$ of the edge-set of a $z$-tree and the edge $ij$ is not uniquely determined
by the $z$-tree. To make the partition unique, take a longest path $P^*$ in a $z$-tree $Z$ whose edges are strictly increasing in length, and
let $ij$ be defined to be the second-to-last edge of the path (see Figure 3). Then $S_i$ is the set of edges $ik$ with $k > j$ and $S_j$ is the set of
edges $hj$ with $h < i$. The edge $ij$ and the sets $S_i, S_j$ are uniquely determined by $Z$, as is the increasing tree $T = Z - S_i - S_j$.
By inspection, a $z$-tree does not contain $P,Q$ or $R$, and we now show the converse:

\begin{thm}
If an ordered tree of interval chromatic number two does not contain $P, Q$ or $R$, then it is a $z$-tree.
\end{thm}

\begin{proof}
We begin with the following observation:
\begin{center}
\parbox{5,6in}{\sl If $T$ is an ordered tree with $\chi_i(T)=2$,  not containing $P,Q$ or $R$, then $T$ contains no
path of length four with at least one pair of crossing edges.}
\end{center}
The following claim is helpful:
\begin{center}
\parbox{5.6in}{\sl If $T$ is an ordered tree with $\chi_i(T)=2$ containing crossing edges $i'j$ and $ij'$ where $i' < i < j < j'$ and $ij \not \in E(T)$, then $T$ contains $P$ or $Q$ or $R$.}
\end{center}

We first prove the claim. Since $T$ is a tree, there exists a path in $T$ whose first and last edges are $i'j$ and $ij'$. If this path has length four, then we have a path of length four with a crossing, a contradiction.
Therefore the path must have length three. Since $\chi_i(T)=2$, and $i' < i < j < j'$,
$i'$ and $i$ are not adjacent and $j$ and $j'$ are not adjacent. Therefore the only possibility is that $i'j' \in E(T)$,
but then the edges $i'j',i'j,ij'$ form a copy of $P$ in $T$, a contradiction. This proves the claim.

\medskip

Now we prove the theorem. Let $Z$ be an ordered tree with $\chi_i(Z)=2$  not containing $P,Q$ or $R$, with
intervals $I< J$. Let $xy$ be an edge of $Z$ such that $x \in I$ and $y \in J$ and $y$ has degree 1
in $Z$ (the case $x \in J$ and $y \in I$ is similar). Then $Z'  = Z - \{y\}$ does not contain $P,Q$ or $R$, so $Z'$ is a $z$-tree. We may write $E(Z') = E(T) \cup S_i \cup S_j$, where
$T$ is an increasing tree with longest edge $ij$ with $i \in I$ and $j \in J$, and $S_i = \{ik : k > j\}$ and $S_j = \{hj : h < i\}$. The edge $ij$
is the second-to-last edge of a path longest $P^*$ in $Z'$ whose edge lengths are increasing. Let $ik \in S_i$ be the last edge of the path where
$k \in J$ and $k > j$ (the case that the last edge is $hj$ with $h < i$ is similar).

\medskip

{\bf Case 1.} {\em The edge $xy$ crosses an edge $ab \neq ik$ in $P^*$ where $a \in I$ and $b \in J$.} In this case, either $x < a < b < y$ or $a < x < b < y$. By the claim, if $x < a < b < y$, then $ay \in E(Z)$, contradicting that $y$ has degree 1 in $Z$. So $a < x < b < y$, and the claim gives $xb \in E(Z')$. Now the path $Q^* \subset P^*$ starting with the edges $yx$, $xb$ and $ba$ and ending with the edge $ik$ has length at least four in $Z$ and has a crossing, which is a contradiction. This completes Case 1.

\medskip

{\bf Case 2.} {\em The edge $xy$ crosses no edge of $P^* - ik$.} If $x > i$ then $T \cup \{xy\}$ is an increasing tree and $Z$ is therefore a $z$-tree. We conclude $x \leq i$. If $x < i$, then there is an edge $xj \in S_j$. But then the path $yxjik$ is a path of
length four in $Z$ with a crossing, a contradiction. We conclude $x = i$, and $y > j$, and now $E(Z) = E(T) \cup S_i' \cup S_j$ where $S_i' = S_i \cup \{iy\}$, so $Z$ is a $z$-tree. \end{proof}

\subsection{Ordered trees with linear extremal function}\label{linear}

This section is devoted to determining the extremal function for $z$-trees, thereby completing the proof of Theorem \ref{trees}.
We determine first the extremal function of increasing trees.

\begin{lem}\label{inc}
Let $T$ be an increasing tree with $k$ edges. Then $\ex_{\rightarrow}(n,T) = (k - 1)n - {k \choose 2}$ for $n \geq k + 1$.
\end{lem}

\begin{proof}
Observe that the longest edge in $T$ has length at least $k$. Therefore the ordered $n$-vertex graph $G^*$
consisting of all edges $ij$ with $i,j \in [n]$ such that $1 \leq |i - j| < k$ cannot contain $T$, and so
\[ \ex_{\rightarrow}(n,T) \geq e(G^*) = \sum_{i = 1}^{k - 1} (n - i) = (k - 1)n - {k \choose 2}.\]
Now we establish equality. Suppose $G$ is an $n$-vertex ordered graph that does not contain $T$. We prove by induction on $k$ that $e(G) \leq (k - 1)n - {k \choose 2}$ for
$n \geq k + 1$. For $k = 1$, this is clear since any single edge is an increasing tree, so $G$ in that case is empty.
Suppose $T$ is an increasing tree with $k + 1$ edges. Let $uv$ be the longest edge of $T$, where $u < v$, and suppose $v$ is a leaf of $T$. Let $T' = T - uv$.
Assuming $V(G) = [n]$, remove for every $i \leq n - k$ in $V(G)$ the longest edge $ij$ with $j > i$. Then the total number of edges removed from $G$ is at most
$n - k$. We therefore obtain an ordered graph $G'$ with at least $e(G') \geq e(G) - (n - k)$ edges. If $G'$ contains the ordered tree $T'$, say $u$ is mapped to
$i \in [n]$, then $i \leq n - k$, so there exists an edge $ij \in E(G) \backslash E(G')$ such that $j$ is larger than any vertex in the embedding of $T'$ in $G$.
Then adding $ij$ we get an embedding of $T' + uv = T$ in $G$, a contradiction. We conclude $G'$ does not contain $T'$, so by induction $e(G') \leq (k - 1)n - {k \choose 2}$. Therefore
\[ e(G) \leq (k - 1)n - {k \choose 2} + (n - k) = kn - {k + 1 \choose 2}.\]
This completes the proof.
\end{proof}

This proof extends to give the extremal function for $z$-trees in a fairly simple way:

\begin{lem} \label{zlem}
Let $Z$ be a $z$-tree with $k$ edges. Then for $n \geq k + 1$,  $\ex_{\rightarrow}(n,Z) = (k - 1)n - {k\choose 2}$.
\end{lem}

\begin{proof}
We may write $Z = T \cup S_i \cup S_j$ where $T$ is an increasing tree with longest edge $ij$ with $i < j$, and $S_i$ consists of edges $ik$
with $k > j$ and $S_j$ consists of edges $hj$ with $h < i$. Suppose $|E(T)| = a$, $|S_j| = b$ and $|S_i| = c$.
We construct an $n$-vertex ordered graph $G^*$ with no copy of $Z$ as follows: the vertex set of $G^*$ is $[n]$, whereas
the edge set consists of $E_a = \{xy : 1 \leq y - x < a\}$, $E_b = \{xy : x \leq b\}$ and $E_c = \{xy : y > n - c\}$.
Let $f(a,b,c) = |E_a \cup E_b \cup E_c|$. A calculation shows $f(a,b,c) = (k - 1)n - {k \choose 2}$. Furthermore,
$E_a \cup E_b \cup E_c$ does not contain a copy of $Z$: since $T$ has $a$ edges, $ij$ has length at least $a$,
so $ij \not \in E_a$. If $ij \in E_b$, then $i \leq b$. However, $Z$ has $b$ vertices preceding $i$, namely the vertices in $S_j$,
so this is not possible. Similarly, if $ij \in E_c$, then $j > n - c$, but since $Z$ has the $c$ vertices in $S_i$ after $j$, this too is impossible. Therefore $G^*$ does not contain $Z$, and we have $\ex_{\rightarrow}(n,Z) \geq (k - 1)n - {k \choose 2}$.

\medskip

We now prove $\ex_{\rightarrow}(n,Z) = f(a,b,c)$ by induction on $|S_i| = c$. If $c = 0$, then Lemma \ref{inc} proves the required equality.
If $c \geq 1$, then we observe $f(a,b,c) - f(a,b,c - 1) = n - k + 1$. Let $G$ be an $n$-vertex ordered graph not containing $Z$.
Following the notation above, with $ij$ the longest edge of $T \subset Z$, for each vertex $g : b < g \leq n - a - c + 1$,
delete the longest edge $gh \in E(G)$ with $h > g$. The number of edges deleted is $n - a - b - c + 1 = n - k + 1$. If this new graph $G'$
contains $Z' = Z - ij'$ where $j'$ is the last vertex of $Z$, then $G$ contains $Z$: we observe $b < i \leq n - a - c + 1$, and so there is a
longest edge $ij' \in E(G) \backslash E(G')$ which can be added to $Z'$ to get $Z$. Therefore $G'$ does not contain $Z'$, and by induction,
$e(G') \leq f(a,b,c-1)$ which implies $e(G) \leq e(G') + n - k + 1 \leq f(a,b,c-1) + n - k + 1 = f(a,b,c)$. This completes the proof.
\end{proof}

\section{Convex geometric trees}
	In this section we prove Theorem~\ref{cgthm}. We denote a crossing four-edge path by the shorter notation crossing $P_4$.	In Section \ref{cghnonlinear}, we give the constructions which show that each $P^i$ and each crossing $P_4$ has extremal function of order at least $n\log n \log n$.
	In Section \ref{cghztrees}, we describe the structure of the cg trees $T$ with $\chi_c(T)=2$ which contain neither a crossing $P_4$ nor any $P^i$. Then in Section~\ref{cghlinear}, we show that these trees have linear extremal function.

\subsection{Convex geometric trees with nonlinear extremal function} \label{cghnonlinear}

We begin by noting that Bra\ss-K\'arolyi-Valtr~\cite{BKV} proved that $\ex_{\cir}(n, P^i) =\Theta(n \log n)$. Actually, they proved this only for $P^0$ but exactly the same proof (both upper and lower bounds)  works for $P^1$ and $P^2$ as well.

In order to present our constructions that avoid crossing $P_4$s, we need a theorem of Tardos~\cite{T2}. The setup of his theorem is as follows. We are given a bipartite graph $G=(A,B,E)$ with a proper edge coloring $c$ with $d$ colors in which the colors
are linearly ordered.

A walk $e_1e_2e_3e_4$ is called {\em fast} if $c(e_2)<c(e_3)<c(e_4)\leq c(e_1)$.
A walk $e_1e_2e_3e_4$ is called {\em slow} if it starts in $B$, $c(e_2)<c(e_3)<c(e_4)$ and $c(e_2)<c(e_1)\leq c(e_4)$.

\begin{thm}[Tardos~\cite{T2} p. 549]\label{c3} Let $G=(A,B,E)$ be
	a bipartite graph with a proper edge coloring with $d$ colors. There exists a subgraph $G'=(A,B,E')$ of $G$ without slow walks and with $|E'|>\frac{\log d}{480d}|E|$.
	Similarly, there exists a subgraph $G''=(A,B,E'')$ of $G$ without fast walks and with $|E''|>\frac{\log d}{480d}|E|$.
\end{thm}

We are now ready to present our constructions for Theorem~\ref{cgthm}

{\bf Construction.}
Let $v_1,\ldots,v_n$ be in clockwise order on $\Omega$ and form the vertex set $V$ of our construction $F_n$, where $n=2^k$. The edge set of $F_n$ consists of $k-1$
matchings $M_1,\ldots,M_{k-1}$, each of size $n/4$. For each $1\leq j\leq k-1$,
$$M_j=\{v_{2i-1}v_{2i-2+2^{j}}\;:\; i\in [n/4]\}.$$
Let $V_1=\{v_1,v_3,\ldots,v_{n-1}\}$ and $V_2=\{v_2,v_4,\ldots,v_{n}\}$.
For every edge $e=v_iv_j$ in $F_n$, if $j<i$, then $j$ is the {\em left end} if $j<i$ and {\em right end} otherwise. Note that 
\begin{center}
\begin{tabular}{lp{5in}}
(i) & $|E(F_n)|=(k-1)n/4=(\log_2n-1)n/4$;\\
(ii) & the left ends of all edges are in $V_1$ and all right ends are in $V_2$;\\
(iii) & $F_n$ does not contain a path $v_{i_1}v_{i_2}v_{i_3}v_{i_4}$ such that $i_2<i_4<i_1<i_3$. 
\end{tabular}
\end{center}

Case (iii) is referred to by Tardos~\cite{T2} as a {\em heavy path}.
We consider $M_1,\ldots,M_{k-1}$ as color classes of an edge coloring $c$ of $F_n$ in which the colors are ordered according to  their indices.

By Theorem~\ref{c3}, $F_n$ contains subgraphs $F_{n,1}$, $F_{n,2}$ and $F_{n,3}$ such that

\begin{center}
\begin{tabular}{lp{5in}}
(P1) & $|F_{n,j}|\geq  \frac{\log (k-1)}{480(k-1)}|E(F_n)|\geq \frac{\log (k-1)}{1920}n$ for each $1\leq j\leq 3$; \\
(P2) & $F_{n,1}$ does not contain fast walks; \\
(P3) & $F_{n,2}$ does not contain slow walks starting in $V_1$; \\
(P4) & $F_{n,3}$ does not contain slow walks starting in $V_2$.
\end{tabular}
\end{center}

We also will use the cg graph $F_{n,0}$ with the same vertex set $V$ and
$E(F_{n,0})=\{v_iv_j: \,1\leq i\leq n/2, n/2+1\leq j\leq n\}$. \qed
\medskip

\begin{definition}\label{defL}
We denote by $L$ the  cg three-edge path with interval chromatic number greater than two. In other words, $L$ is the $3$-edge cg path
$xyzu$ such that $x<y<z<u$ (in the cyclic ordering).
\end{definition}

{\bf Remark.}
The cg graph $F_{n,0}$ does not contain $L$.

\begin{definition}
A path $P=x_1x_2\ldots x_s$ in a cgg is a {\em zigzag} if it has no crossing and for every $2\leq j\leq s-2$, the sets
$\{x_1,\ldots,x_{j-1}\}$ and $\{x_{j+2},\ldots,x_s\}$ are on different sides of the chord $x_jx_{j+1}$. Alternatively, $P$ has no crossing and $\chi_c(P)=2$.
\end{definition}

Perles (see ~\cite{MP} p. 292) proved that $\ex_{\cir}(n, P) = O(n)$ for any zigzag path $P$.   Our main result is that the construction presented above contains no copy of a crossing $P_4$.

\begin{thm} \label{cla3}
	For every four-edge path $P$ in a cgg apart from the zigzag path, $\ex_{\cir}(n, P) = \Omega(n \log \log n)$. Moreover, $P$ is not contained in one of $F_{n,j}$ for $0\leq j\leq 3$.
\end{thm}

\begin{proof} Since each $F_{n,j}$ has at least $\Omega(n \log k) = \Omega(n \log\log n)$ edges, it suffices to show that for each crossing four-edge path $P$ there is a $j$ for which $P \not\subset F_{n,j}$. Every type of a cgg four-edge path corresponds to a cyclic permutation of $[5]$. So we need to consider $4!=24$ types of them.
By the remark above, it is enough to consider the types with  cyclic chromatic number $2$. Suppose such a path
(i.e. with cyclic chromatic number $2$ and distinct from a zigzag path)
$P=abcdf$ can be
embedded into $F_{n,1}$. Suppose that for $x\in\{a,b,c,d,f\}$, $x$ is mapped onto $v_{i_x}$.

Assume $i_c$ is odd (the proof for even $i_c$ will be symmetric). By the structure of $F_n$, $i_b>i_c$ and $i_d>i_c$.
So by symmetry we may suppose
\begin{equation}\label{e1}
i_c<i_d<i_b.
\end{equation}
Again by the structure of $F_n$, $i_b>i_a$. If $i_d<i_a<i_b$, then the cyclic chromatic number of the path $a,b,c,d$ is $3$, a contradiction. The situation $i_c<i_a<i_d$ impossible because the lengths of the edges are of the form $2^j-1$. So, the only possibility
is $i_a<i_c<i_d<i_b$.

Once again by the structure of $F_n$, $i_d>i_f$. If $i_c<i_f<i_d$, then $P$ is a zigzag path, contradicting our choice.
If $i_f<i_a$, then $i_d-i_f\geq i_b-i_a$, since otherwise $2(i_d-i_f)< i_b-i_a$ and $2(i_b-i_c)< i_b-i_a$, a contradiction.
But in this case, $F_{n,1}$ contains the fast walk $fdcba$ contradicting (P2). Thus, $i_a<i_f<i_c$. This means we need to consider
only cyclical structure $(1,5,3,4,2)$ and (when we switch from the case of odd $i_c$ to the even) $(5,1,3,2,4)$.

\medskip
{\bf Case 1:} $(1,5,3,4,2)$. We claim that $F_{n,2}$ does not contain it. Indeed, suppose it does. If $i_c$ is odd, then repeating the
above argument we come to  $i_a<i_f<i_c$. But this means $F_{n,2}$ contains the slow walk $fdcba$ contradicting (P3).
Thus assume $i_c$ is even. Then by the structure of $F_n$, $i_d<i_c$ and $i_b<i_c$, say  $i_d<i_b<i_c$. Then by the cyclic
structure of our path, $i_b<i_f<i_c$. This is impossible, since $2(i_f-i_d)< i_c-i_d$ and $2(i_c-i_b)< i_c-i_d$.

\medskip
{\bf Case 2:} $(5,1,3,2,4)$. We claim that $F_{n,3}$ does not contain this path. The proof is symmetric to Case 1.\end{proof}

\subsection{Structure of trees avoiding $\cal P$ and a crossing $P_4$} \label{cghztrees}

We need the definitions of increasing trees and $z$-trees in the cg setting.

{\bf Convex geometric increasing trees.} A {\em cg increasing tree} is a cg tree of cyclic chromatic number two obtained as follows. Start with an (ordered) increasing tree with vertex set $[n]$ and
view the linear ordering of the vertices as a cyclic ordering  $n<n-1<\cdots <2<1$.
 Note that a cg increasing tree has no crossing edges, and the edges
have a unique ordering by the increasing order of their lengths (when viewed in terms of the natural linear order on $[n]$).

\medskip

{\bf Convex geometric $z$-trees.} A {\em cg $z$-tree} is a cg tree $Z$ with cyclic chromatic number two, obtained from a $z$-tree with ordered vertex set $[n]$ and intervals $I, J \subset [n]$ by viewing the linear ordering of the vertices as a cyclic ordering $n<n-1<\cdots <2<1$.
Note that $Z$ is a union of a cg increasing tree $T$ with longest edge $ij$ where $i \in I$ and $j \in J$, together with a set $S_j$ of edges of the form $hj$ with $h \in I$ and $i<h$ and a set $S_i$ of edges of the form $ik$ with $k \in J$ and $k < j$ (See Figure 4 viewed as a cyclic ordering). These sets $S_i$ and $S_j$ are allowed to be empty.
Note that any two of the edges $hj$ and $ik$ cross (see Figure 3 for an example).


%
%

Our main structural result is the following.

\begin{thm} \label{structure}
	Let $T$ be a cg tree with $\chi_c(T)=2$ that contains no crossing $P_4$ and no member of ${\cal P}$. Then $T$ is a cg $z$-tree.
	\end{thm}

\begin{proof}
	Suppose we have an embedding of $T$ in a circle $\Omega$. We will simultaneously refer to $T$ as well as to the geometric properties of its embedding.  Say that an edge $e$ is {\em heavy} if both its endpoints have a neighbor on the same side of $e$. Since we have no $L$ (see Definition~\ref{defL}), this  means that every heavy edge $e$ gives rise to a crossing three-edge path with central edge $e$.
	
	{\bf Case 1.} There is a heavy edge $e=i'j$.	
	Suppose $j'$ is a neighbor of $i'$ and $i$ is a neighbor of $j$ such that both $i$ and $j'$ are on the same side of $e$, assume by symmetry that $j'$ and $i$ lie on the   arc $(j,i')$ of $\Omega$ taken clockwise. Then  $i<i'<j<j'$ otherwise we obtain $L$.
	This shows that all such neighbors of $j$ lie clockwise of all such neighbors of $i'$ in $(j, i')$ and we obtain a double star as shown below. Moreover, each of these neighbors $j', i$ has degree one otherwise we obtain a crossing $P_4$.
		\begin{center}
		\includegraphics[width=1.6in]{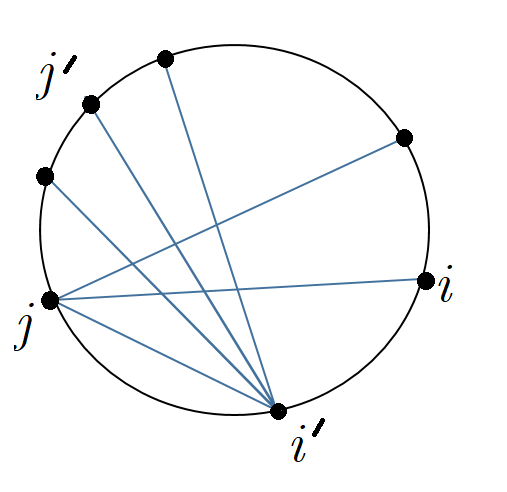}\\
Figure 4. 
			\end{center}
	Hence, to grow $T$ further, we must consider neighbors of $i'$ or $j$ on the   arc $(i', j)$ of $\Omega$ which omits $j'$ and $i$ (see Figure 4).
		We now claim that on the   arc $(i',j)$, the tree $T$ is an increasing tree and moreover, there is no edge that crosses $e=i'j$. This will complete the proof in this case.
		First observe that $i'$ and $j$ cannot both have neighbors in $(i', j)$ otherwise we get $L$ as before or $P^2$ as shown below.
			\begin{center}
			\includegraphics[width=1.9in]{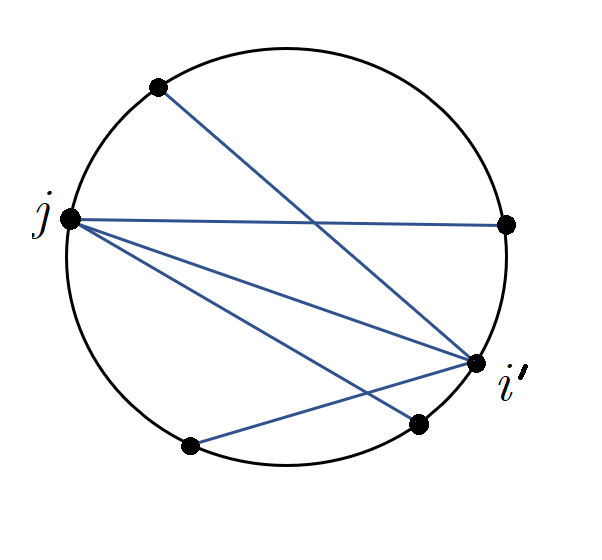} \\
Figure 5.
		\end{center}
	So we may assume by symmetry that only $j$ has neighbors in  $(i',j)$.	
	 If $j_1,j_2,\dots,j_r$ are the neighbors of $j$ in $(i',j)$ in
	increasing clockwise order, then only $j_r$ can have degree at least two, otherwise we get a copy of $L$ or $P^1$; furthermore,  all neighbors of $j_r$ are in the segment $(j_r,j)$; otherwise we have $L$ or a crossing $P_4$.
	We now continue in this way: if $j_r$ has neighbors $k_1,k_2,\dots,k_m,j$ in increasing clockwise order in $(j_r,j)$,
	then only $k_1$ can have degree at least two, otherwise we have  a copy of $L$, or $P^0$, or a crossing $P_4$. Furthermore, all neighbors of $k_1$ are in the segment $(j_r,k_1)$ otherwise we obtain  $L$ or a crossing $P_4$. We continue this process till we exhaust all of $T$.
	
	{\bf Case 2.} There is no heavy edge.
	In this case we claim the stronger statement that $T$ is cg increasing tree. Start by choosing any edge $e=i'j$ and consider the neighbors of $i'$ or of $j$ in the   arc $(i', j)$ (traversed clockwise as usual). Since $e$ is not heavy, at most one of $i',j$ has neighbors in   $(i',j)$, say $j$. Then, using the fact that there is no heavy edge, we proceed as in the previous paragraph until we have exhausted all vertices of $T$ in $(i', j)$. Then we repeat this argument in the   arc $(j, i')$ to show that $T$ is a cg increasing tree.
		\end{proof}

\subsection{Convex geometric trees with linear extremal function} \label{cghlinear}

We are now in a position to complete the proof of Theorem~\ref{cgthm}

{\bf Proof of Theorem~\ref{cgthm}.} By Theorems~\ref{cla3} and~\ref{structure}, it suffices to show that every cg $z$-tree $Z$ with $k\ge 2$ edges satisfies $\ex_{\cir}(n, Z) \le 2(k-1)n$.
We will prove this by induction on $k$, with the case $k=2$ being trivial. For the induction step, suppose that $Z$ is a cg $z$-tree with $k>2$ edges and $G$ is an $n$-vertex cg graph with more than $2(k-1)n$ edges.

 Let us first do the case when $Z$ is a double star with edge set
$S_i \cup S_j$. In this case, we view $\Omega$ as a linearly ordered set, and use the fact that
$\ex_{\cir}(n, Z) \le \ex_{\to}(n, Z)\le (k-1)n$ where  the bound $\ex_{\to}(n, Z)\le (k-1)n$ follows from Lemma~\ref{zlem}, with $Z$ viewed as the appropriate ordered double star.

We now assume that $Z$ has a leaf $x$ incident  to some shortest  edge outside $S_i \cup S_j$ and let $Z'=Z-v$ (in Figure 3, this corresponds to deleting the lowest solid edge). Let $G'$ be the cg graph obtained from $G$ by deleting, for each vertex $v \in V(G)$, the two shortest edges incident to $v$, one in each direction. We delete at most $2n$ edges, so $G'$ has more than $2(k-2)n$ edges and hence by induction $G'$  contains a copy of $Z'$. We then extend this copy of $Z'$ to a copy of $Z$ in $G$ using one of the edges that was deleted in forming $G'$.
\qed

\paragraph{Acknowledgement.}
This research was partly conducted during AIM SQuaRes (Structured Quartet Research Ensembles) workshops, and we gratefully acknowledge the support of AIM.

{\small

\begin{tabular}{ll}
\begin{tabular}{l}
{\sc Zolt\'an F\" uredi} \\
Alfr\' ed R\' enyi Institute of Mathematics \\
Hungarian Academy of Sciences \\
Re\'{a}ltanoda utca 13-15. \\
H-1053, Budapest, Hungary. \\
E-mail:  \texttt{zfuredi@gmail.com}.
\end{tabular}
& 
\begin{tabular}{l}
{\sc Alexandr Kostochka} \\
University of Illinois at Urbana--Champaign \\
Urbana, IL 61801 \\
and Sobolev Institute of Mathematics \\
Novosibirsk 630090, Russia. \\
E-mail: \texttt {kostochk@math.uiuc.edu}.
\end{tabular} \\ \\

 \begin{tabular}{l}
{\sc Dhruv Mubayi} \\
Department of Mathematics, Statistics \\
and Computer Science \\
University of Illinois at Chicago \\
Chicago, IL 60607. \\
\texttt{E-mail: mubayi@uic.edu}.
\end{tabular} 
&\begin{tabular}{l}
{\sc Jacques Verstra\"ete} \\
Department of Mathematics \\
University of California at San Diego \\
9500 Gilman Drive, La Jolla, California 92093-0112, USA. \\
E-mail: {\tt jverstra@math.ucsd.edu.}
\end{tabular}
\end{tabular}
}

\end{document}